%% file: ProjectiveDelineability.tex
\documentclass[conference]{IEEEtran}
\IEEEoverridecommandlockouts
\usepackage{amsmath,amssymb,amsfonts}
\usepackage{algorithmic}
\usepackage{graphicx}
\usepackage{textcomp}
\usepackage{xcolor}

\usepackage{graphicx} 
\usepackage{xcolor}
\usepackage{amssymb}
\usepackage{amsmath}
\usepackage{amsthm}
\usepackage{lmodern}
\usepackage[T1]{fontenc}
\usepackage{hyperref}

\usepackage{cleveref}

\newcommand{\R}{\mathbb{R}}
\newcommand{\RP}{\mathbb{RP}^1}

\newcommand{\N}{\mathbb{N}}

\newcommand{\Disc}{\text{Disc}}

\newcommand{\Res}{\text{Res}}
\newcommand{\ord}{\text{ord}}

\newcommand{\id}{\mathrm{id}}
\newcommand{\AP}{A^*P}

\newcommand{\Rstar}{\mathbb{R}^2_*}

\theoremstyle{definition}
\newtheorem{theorem}{Theorem}
\newtheorem{example}{Example}
\newtheorem{proposition}{Proposition}
\newtheorem{lemma}{Lemma}
\newtheorem{definition}{Definition}
\newtheorem{remark}{Remark}

\newtheorem{corollary}{Corollary}

\usepackage{tikz}
\usepackage{pgfplots}
\usepackage{rwthcolors}
\newenvironment{myaxis}[1][ymin=-2.2,ymax=2.2,xmin=-2.2,xmax=2.2,x=1.25cm,y=1.25cm]{ 
  \begin{axis}[axis on top=false,
    axis line style=black,
    xlabel=$\textcolor{black!100}{x_1}$, ylabel=$\textcolor{black!100}{x_2}$,grid=none,
    xtick=\empty,
    ytick=\empty,
    yticklabels=\empty,
    xticklabels=\empty,
    #1]
}{\end{axis}}

\newenvironment{myaxis2}[1][ymin=-4.4,ymax=4.4,xmin=-4.4,xmax=4.4,x=.625cm,y=.625cm]{ 
  \begin{axis}[axis on top=false,
    axis line style=black,
    xlabel=$\textcolor{black!100}{x_1}$, ylabel=$\textcolor{black!100}{x_2}$,grid=none,
    xtick={-1,0,1},
    ytick={-1,0,1},
    yticklabels={-1,0,1},
    xticklabels={-1,0,1},
    #1]
}{\end{axis}}

\makeatletter
\newcommand{\pgfplotsdrawaxis}{\pgfplots@draw@axis}
\makeatother
\pgfplotsset{only axis on top/.style={axis on top=false, after end axis/.code={
             \pgfplotsset{axis line style=opaque, ticklabel style=opaque, tick style=opaque,
                          grid=none}\pgfplotsdrawaxis}}}
\pgfplotsset{compat = newest}

\begin{document}

\title{On Projective Delineability

\author{
    \IEEEauthorblockN{Lucas Michel\IEEEauthorrefmark{1},
    Jasper Nalbach\IEEEauthorrefmark{2}, 
    Pierre Mathonet\IEEEauthorrefmark{1},
    Naïm Zénaïdi\IEEEauthorrefmark{1},}
    \IEEEauthorblockN{
    Christopher W. Brown\IEEEauthorrefmark{3},
    Erika \'Abrah\'am\IEEEauthorrefmark{2},
    James H. Davenport\IEEEauthorrefmark{4},
    Matthew England\IEEEauthorrefmark{5}}
    \IEEEauthorblockA{}
    \IEEEauthorblockA{\IEEEauthorrefmark{1}University of Liège, \IEEEauthorrefmark{2}RWTH Aachen University,}
    \IEEEauthorblockA{\IEEEauthorrefmark{3}United States Naval Academy, \IEEEauthorrefmark{4}University of Bath, \IEEEauthorrefmark{5}Coventry University}
}
\thanks{Contact : lucas.michel@uliege.be.}%
\thanks{ P.~Mathonet, L.~Michel and N.~Zenaïdi are supported by the FNRS-DFG PDR Weaves (SMT-ART) grant 40019202. E.~Ábrahám and J.~Nalbach are supported by the Deutsche Forschungsgemeinschaft (DFG, German Research Foundation) as part of AB 461/9-1 \emph{SMT-ART}. J.~Nalbach is supported by the DFG as part of RTG 2236 \emph{UnRAVeL}. M.~England and J.~Davenport are supported by the UKRI EPSRC DEWCAD Project (grant EP/T015748/1 and EP/T015713/1 respectively). J.~Davenport is funded by the 
DFG under Germany's Excellence Strategy (EXC-2047/1-390685813). This publication is based upon work from COST Action EuroProofNet, supported by COST (European Cooperation in Science and Technology, www.cost.eu). 
This is the author's version of the work.}
}

\maketitle

\begin{abstract}
We consider cylindrical algebraic decomposition (CAD) and the key concept of delineability which underpins CAD theory.  We introduce the novel concept of projective delineability which is easier to guarantee computationally.  We prove results about this which can allow reduced CAD computations.
\end{abstract}


\input{intro}

\input{preliminaries}

\input{proj-del}

\end{document}

%% file: intro.tex
\section{Introduction}

\emph{Cylindrical algebraic decomposition (CAD)} is a well-known tool for computing with polynomial constraints over the reals, developed in the 1970s by Collins \cite{collins1975} as the basis for quantifier elimination (QE) over real-closed fields. CAD has been continuously researched, improved, and applied in the half century since its inception. 

Recently, CAD theory has been adapted and specialized for \emph{satisfiability modulo theories} (SMT) solvers which check the satisfiability of Boolean combinations of real polynomial constraints: e.g. the \emph{NuCAD} \cite{nucad}, the \emph{NLSAT} \cite{jovanovic2012}, and the \emph{CAlC} \cite{abraham2021} algorithms.  
The corresponding \verb+QF_NRA+ category of the SMT-LIB benchmarks contains instances of such problems originating from e.g. verification systems, the life sciences and economics.  

Central to CAD is the notion of \emph{delineability} of a polynomial over a connected set (\emph{cell}). Roughly speaking, an $n$-variate polynomial $P$ is delineable over a cell $S \subseteq \R^{n-1}$ if the set of zeros of $P$ in the cylinder $S\times\R$ is the disjoint union of the graphs of continuous real-valued functions of $n-1$ variables (we call these \emph{real root functions}).  These graphs cut the cylinder $S\times\R$ into cells which the sign of $P$ is invariant within. We order these cells bottom-to-top to form a \emph{stack}. The key observation is that the number of cells, their ordering, and the sign of $P$ within cells, are all the same no matter which point of $s \in S$ we are over.  Thus the presence of delineability allows for construction at a sample point to be concluded valid over the cell.

We are particularly interested by the \emph{single cell construction} \cite{NASBDE24}, which is used as a subroutine by the three SMT algorithms mentioned earlier \cite{nucad, jovanovic2012}, \cite{abraham2021}.  This uses CAD projection theory to compute a single CAD cell around a given sample point within which a set of polynomials is sign-invariant. Compared to traditional CAD algorithms, the single cell construction allows us to use the sample point to detect which projection polynomials are not relevant for the cell under construction, and omit them in the computation. 

In this paper we introduce the new notion of \emph{projective delineability}: easier to  guarantee computationally than delineability but still useful.  Our new notion allows us to further omit leading coefficients during single cell construction in cases where projective delineability of a polynomial is sufficient instead of its delineability.  Indeed, projective delineability is equivalent to delineability in the case where the leading coefficient of $P$ never vanishes, but may still be guaranteed when the leading coefficient has zeroes (see Proposition \ref{prop:linkDelProjdel} for a more precise statement). 

\section{Delineability and Leading Coefficients}

A CAD is a decomposition of $\R^n$, consisting of cells which maintain properties for the input; usually that a finite family of input polynomials $\mathcal{F}$ in $n$ variables are sign-invariant within each cell. These cells are \emph{cylindrically arranged}, meaning that the projection of a CAD of $\R^n$ is a CAD of $\R^{n-1}$, and over each cell $S$ of the latter, the cylinder $S \times \R$ is decomposed by cells of the former. 
A CAD may be computed recursively: for obtaining the CAD of $\R^n$, we compute a \emph{projection} $\mathcal{F}'$ in $n-1$ variables of $\mathcal{F}$ such that $\mathcal{F}'$ defines a suitable CAD of $\R^{n-1}$. The following notion formalizes the property that the CAD of $\R^{n-1}$ needs to fulfil if we are to extend it to a CAD of $\R^n$.

\begin{definition}[\cite{collins1975, McCallum1988}]\label{def:del}
    We say that $P \in \R[\textbf{x};x_n]$ is \emph{delineable} on a subset $S$ of $\R^{n-1}$ if there exist $k \in \N$ and some continuous functions $\theta_1, \ldots, \theta_k: S \to \R$ such that:
        \begin{enumerate}
            \item the set of zeros of $P$ in $S \times \R$ is the disjoint union $\text{Graph}(\theta_1) \sqcup \ldots \sqcup \text{Graph}(\theta_k)$;
            \item for all $l \in \{1,\ldots, k\},$ there exists $m_l \in \N^*$ such that for all $\textbf{x} \in S$, the multiplicity of the root $\theta_l(\textbf{x})$ is $m_l$. 
        \end{enumerate}
        We say that $\theta_1,\ldots,\theta_k$ are the \emph{real root functions} of $P$ over $S$. 
        The graph of $\theta_l$ is called a real $P$-\emph{section} over $S$.
\end{definition}

The computation of the projection $\mathcal{F}'$ has been an active area of research. Our work is based on
the projection operator of McCallum \cite{McCallum1998} with the optimization by Brown \cite{Brown2001}: $\mathcal{F}'$ contains  the discriminant and leading coefficient of each $P \in \mathcal{F}$ and the resultant of each pair $P,Q \in \mathcal{F}$.

The new notion of projective delineability is designed for the \emph{single cell construction} \cite{NASBDE24} discussed above.  
Here the cell is already constructed using a partial CAD projection (given the sample point we identify that some projection polynomials for the full CAD are not needed for the region of interest). Prior work reduced the amount of resultants and some discriminants. Projective delineability allows detecting irrelevant leading coefficients also.

\begin{example}
    \label{ex:scc}

    Consider the polynomials $P = x_1^2+x_2^2-1$ and $Q = x_1x_2-1$ and the sample $s=(-\tfrac{1}{2},0)$ as depicted in \Cref{fig:scc-ex} (left). 
    The darker shaded region around the sample is a CAD cell bounded by the roots of $P$ in the $x_2$-axis. The bounds on the $x_1$-axis are defined by the projection: using regular delineability means $x_1$ is bounded by the discriminant of $P$ from below and the leading coefficient of $Q$ from above. However, this upper bound on $x_1$ is too tight, since the singularity of $Q$ does not affect the sign-invariance property of the cell. We could thus extend the single cell to include also the lighter shaded region.
    
    In this paper, we formalize what ``\emph{delineability without leading coefficient}'' means. 
    To do so, we view the roots of a polynomial in real projective space and discover that these roots behave nicely, i.e. we can describe the variety of a polynomial using continuous functions (see \Cref{fig:scc-ex}, right). 
    Projective delineability of $Q$ and the order-invariance of the resultant of the two polynomials is sufficient to show that no real root of $Q$ can enter or pop-up within the larger cell. We can thus safely leave out the leading coefficient of $Q$ in this example.

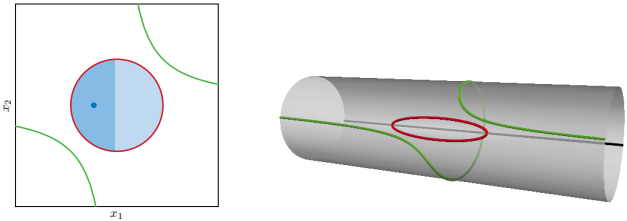
\begin{figure}[h]
    \centering
    \begin{tikzpicture}[scale=0.49]
        \begin{myaxis}
            \fill[blue25] (0,0) circle (1);
            \begin{scope}
                \clip (-1,-1) rectangle (-.05,1);
                \fill[blue50] (0,0) circle (1);
            \end{scope} 
            
            \draw[very thick, color=green100,domain=-2.5:-0.2, variable=\x] plot ({\x},{1/\x});
            \draw[very thick, color=green100,domain=0.2:2.5, variable=\x] plot ({\x},{1/\x});
            \draw[very thick, color=red100] (0,0) circle (1);

            \draw[blue100,fill] (-.5,0) circle (0.05);
        \end{myaxis}

        \node at (12,2) {\includegraphics[width=15em,trim={1cm 3cm 4cm 3cm},clip]{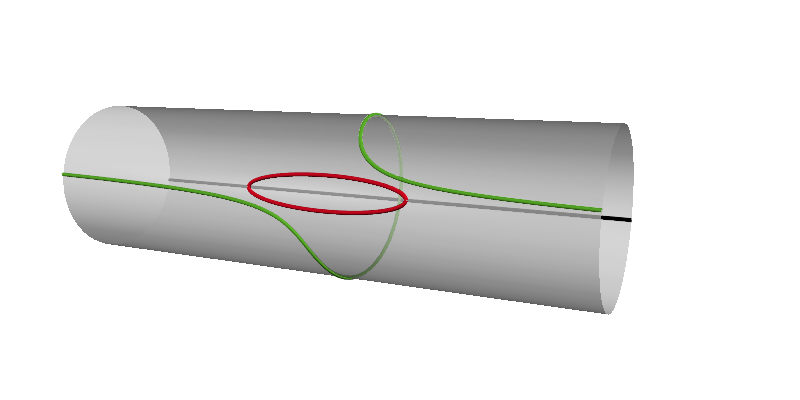}};
    \end{tikzpicture}

    \vspace{-1em}
    \caption{The sets of roots of  $P$ (red) and $Q$ (green) seen in $\R^2$ (left) and in $\R \times \RP$ (right)}
    \label{fig:scc-ex}
\end{figure}
\end{example}

%% file: preliminaries.tex
\section{Preliminaries}\label{Sec:basics}

\subsection{The projective line}\label{sec:projline}

The points of the \emph{real projective line $\RP$} are the lines of $\R^2$ passing through $(0,0)$. Such a line is determined by any of its nonzero vectors $(x,y)\in\Rstar=\R^2\setminus\{(0,0)\}$, and so $\RP$ is equivalently defined as the quotient set of $\Rstar$ by the equivalence relation given by proportionality. We denote the quotient map by $\pi:\Rstar \to \RP$ and the equivalence class of $(x, y)$ by $\pi(x,y)$ or $(x:y)$.

The topology of $\RP$ is the usual quotient topology: $U \subset \RP$ is open if $\pi^{-1}(U)$ is open in $\Rstar$. The map $\psi(x:y)=\frac{y}{x^2+y^2}(x,y)$ is a homeomorphism from $\RP$ to a circle.

The real line $\R$ is usually identified with an open and dense subset of $\RP$  by first identifying $\R$ to a (parameterized) line of $\Rstar$ and then composing with the projection: 
\begin{equation}\label{eqn:iota}
    \iota : \R \to \R^2 : x \mapsto (x,1),\quad \varphi: \mathbb{R}\to \RP : x\mapsto (x:1).
\end{equation}
The only point in $\RP$ that is not in the range of the continuous injective map $\varphi$ is the point $(1:0)$.  For this reason it is sometimes denoted by $\infty$. 
The inverse map 
\[ \varphi^{-1}: \RP\setminus\{\infty\}\to\mathbb{R}\ : (x:y)\mapsto \frac{x}{y}\]
is also continuous.
The choice to identify $\R$ with a given line in $\R^2$ (not containing $(0,0)$) is arbitrary: any other choice of parameterized line reads $\iota_A(\R)$ where $A=(a_{ij})_{i,j\in\{1,2\}}$ is a nonsingular matrix ($A\in GL(2,\R)$) and 
\[
\iota_A : \mathbb{R}\to\mathbb{R}^2:x\to A\begin{pmatrix}x\\1\end{pmatrix}.
\]
Setting $\varphi_A = \pi \circ \iota_A$, the point at infinity (with respect to $A$) is the point that is not in the range of $\varphi_A$, and thus corresponds to the first column of $A$. The maps $(\varphi_A)^{-1}$, for $A \in GL(2, \R),$ endow $\RP$ with the structure of an analytic manifold.  
Note that we have $\varphi_A=\mu_A\circ\varphi$, where for any $A \in GL(2,\R)$, the map $\mu_A$ is defined by
\[
\mu_A : \RP \to \RP : (x:y) \mapsto \pi \left( A\begin{pmatrix}x\\y\end{pmatrix} \right).
\]

\subsection{Binary forms, homogenizations and pull-backs}

Later we deal with roots of polynomials at $\infty$. The framework of binary forms allows us to deal with these roots at infinity on an equal footing with the real roots.  
\begin{definition}\label{def:binaryforms}
A \emph{binary form} of degree $d$ is a  function $g : \R^2 \to \R$ expressed as 
\begin{equation}\label{eq:g}g(x,y) = \sum_{k=0}^d c_kx^ky^{d-k},\quad c_0,\ldots,c_d\in\mathbb{R}.\end{equation}
The set of binary forms of degree $d$ is denoted by $\mathbb{R}[x,y]_d$.
\end{definition}
In view of \Cref{def:binaryforms}, the space $\mathbb{R}[x,y]_d$ is clearly isomorphic to the vector space $\mathbb{R}[x]_{\leqslant d}$ of univariate polynomial functions of degree less than or equal to $d$. An isomorphism is given by the pull-back and its inverse, the homogenization with respect to degree $d$. 
\begin{definition}\label{def:pullback} The \emph{pull back of binary forms} of degree $d\in\N$ is the map
$\iota^{*d}:\mathbb{R}[x,y]_d\to\mathbb{R}[x]_{\leqslant d}:g\mapsto g\circ \iota$.
The \emph{homogenization} with respect to degree $d$ is $H^d=(\iota^{*d})^{-1}$.
\end{definition}

\begin{remark}\label{rem:restrUni} \mbox{}
\begin{enumerate}
\item If $g$ is a binary form of degree $d$ given by \eqref{eq:g}, then $\iota^{*d}g(x)=\sum_{k=0}^d c_kx^k$. It has degree less than or equal to $p < d$ if and only if $c_d=\cdots=c_{p+1}=0$.
\item We have $H^d(P)(x,y)=y^dP(\frac{x}{y})$ for every $P\in \R[x]_{\leqslant d}$ and $(x,y)\in\mathbb{R}^2$ such that $y\neq 0$.
\item By definition, pull-back and homogenization with respect to a same fixed degree are the inverse of each other. This fails to be true if they are not performed with respect to the same fixed degree. For instance, if $g(x,y)=x^2y+xy^2+y^3$, then $\iota^{*3}g(x)=x^2+x+1$, and we have  $H^2(\iota^{*3}g)(x,y)\neq g(x,y)$.
\end{enumerate}
\end{remark}

As we continue it will be useful to consider the pull-back and homogenization with respect to other lines in $\R^2$. 
\begin{definition}\label{def:pullbackgen}
For every $A\in GL(2,\mathbb{R})$, the pull-back with respect to $A$ is $\iota_A^{*d}:\mathbb{R}[x,y]_d\to\mathbb{R}[x]_{\leqslant d}:g\mapsto g\circ \iota_A$ and the corresponding homogenization is $H_A^d=(\iota_A^{*d})^{-1}$.
Finally, we set $A^{*d} : \mathbb{R}[x]_{\leqslant d} \to \mathbb{R}[x]_{\leqslant d} : P \mapsto \iota_A^{*d}(H^d(P))$.
\end{definition}

\begin{remark}
If $P(x)=\sum_{k=0}^dc_kx^k$, we compute
\[
A^{*d}P(x)=\sum_{k=0}^dc_k(a_{11}x+a_{12})^k(a_{21}x+a_{22})^{d-k}
\]
for $A\in GL(2,\R)$. When $a_{21}x+a_{22} \neq 0$, this expression also reads $(a_{21}x+a_{22})^dP\left(\frac{a_{11}x+a_{12}}{a_{21}x+a_{22}}\right)$.
\end{remark}

The next lemma gathers useful formulas to deal efficiently with these operators.
We denote by $R_A$ the right composition with $A$ (seen as an operator from $\R^2$ to $\R^2$), namely $R_A:\mathbb{R}[x,y]_d\to \mathbb{R}[x,y]_d:g\mapsto g\circ A.$

\begin{lemma}\label{lemma:pullback}
For every $d\in\mathbb{N}$ and $A,B\in GL(2,\R)$, we have $\iota_A^{*d}\circ H^d_B=\iota^{*d}\circ R_{B^{-1}A}\circ H^d$. In particular, we have $\iota_A^{*d}=\iota^{*d}\circ R_A$, and $H^d_A=R_{A^{-1}}\circ H^d.$
\end{lemma}

\begin{proof}
It is sufficient to prove the particular cases. The first
one follows from the very definition of $\iota_A^{*d}$, while the second is obtained by inversion.
\end{proof}

\subsection{Projective roots, resultants and discriminants}\label{sec:projroots}

\noindent Projective\footnote{Here and throughout this work, by projective, we mean in $\RP$.} roots are defined for binary forms and then, via homogenization, for univariate polynomials. 
\begin{definition}
A point $(x:y)\in\RP$ is a \emph{projective root} of $g\in\R[x,y]_d$ if $g(x,y)=0$. It is a projective root of $P\in\R[x]_{\leqslant d}$, (w.r.t. $d$) if it is a projective root of $H^d(P)$.
\end{definition}
Note that this definition is valid since, by homogeneity, when $g$ vanishes at a point $(x,y) \in \Rstar$, it vanishes on the whole line $(x:y)$.
We can define the multiplicity of the roots of a binary form using the following result\footnote{See also Lemma 6, $\S$7, of \cite[Chapter 8]{Cox}.}. 
\begin{lemma}\label{lemma:factorization}
All nonzero $g \in \R[x,y]_d$ factor uniquely\footnote{Up to permutation.} as 
{\small\begin{equation}\label{eq:factorization1}
g(x,y) = y^{d-\deg(\iota^{*d} g)} (x-x_1y)^{m_1} \ldots (x-x_ty)^{m_t} Q(x,y),
\end{equation}} 
where $x_1,\ldots,x_t$ are the distinct roots of $\iota^{*d} g$ with multiplicities $m_1,\ldots,m_t$ and $Q$ is a binary form without projective roots.
\end{lemma}

\begin{proof}
If $g \in \R[x,y]_d$ is not zero, then $\iota^{*d} g \in \R[x]_{\leqslant d}$ factors as follows, where $x_1,\ldots,x_t\in\R$ are the distinct roots of $\iota^{*d}g$ and $R$ has no real root:
\begin{equation}\label{eq:ioastarg}
(\iota^{*d} g)(x) = (x-x_1)^{m_1} \ldots (x-x_t)^{m_t} R(x).
\end{equation}
We obtain the decomposition (with $Q= H^{\deg(R)}(R)$) and the properties of the multiplicities by computing $g(x,y)=H^d(\iota^{*d} g)(x,y)$ for $y \neq 0$ (using Remark \ref{rem:restrUni}) and extending the equality to $\R^2$ by continuity. The uniqueness directly follows from the uniqueness of the factorization of $\iota^{*d}g$.
\end{proof}
To sum up, the projective roots of a binary form $g$ are on the one hand the roots of $\iota^{*d}g$, seen as points in $\RP$ through the map $\varphi$, with the same multiplicities (these are called real roots); and on the other hand $\infty=(1:0)$ (the root at infinity) with multiplicity $d-\deg(\iota^{*d}g)$ (when this number is positive). The projective roots with respect to degree $d$ of a univariate polynomial $P\in\R[x]_{\leqslant d}$ are its real roots, seen as elements of $\RP$ via $\varphi$ and the root at infinity with multiplicity $d-\deg(P)$. In particular, $P$ always has a projective root at infinity when $d > \deg(P)$.

\begin{lemma}\label{lemma:linkMult}
If $g \in \R[x,y]_d$ and $A\in GL(2,\R)$, then $(x:y)$ is a root of multiplicity $m$ of $R_A g$ if and only if $\mu_A(x:y)$ is a root of multiplicity $m$ of $g$. The same result holds for the projective roots with respect to degree $d$ of the polynomials $A^{*d}P$ and $P\in\R[x]_{\leqslant d}$, respectively.
\end{lemma}

\begin{proof}
In view of Lemma \ref{lemma:factorization}, the binary form $g$ factors as
\[
g(x,y) =  \prod_{i=0}^tl_i(x,y)^{m_i}Q(x,y) 
\]
and the projective roots of $g$ are the roots of the linearly independent linear forms $l_i$, with multiplicities $m_i$ for $i \in \{0, \ldots, t\}$, while $Q$ has no projective root. We then deduce the factorization of $R_Ag$:
\begin{align*}
R_Ag(x,y) &= \prod_{i=0}^t (R_Al_i(x,y))^{m_i} R_AQ(x,y).
\end{align*}
If $(x_0:y_0)$ is a projective root of $R_Ag$, then $\mu_A(x_0:y_0)$ is a projective root of one, and only one, linear form $l_i$. Hence it is a projective root of $g$ with multiplicity $m_i$. The projective roots of $A^{*d}P$ with respect to $d$ are the projective roots of $H^d(A^{*d}P)=H^d\circ\iota_A^{*d}\circ H^d(P)=R_A(H^d(P))$ by Lemma \ref{lemma:pullback} and the result follows.
\end{proof} 
To emphasize that resultants and discriminants are computed with respect to fixed degree we use the notations $\Res^{p,q}(P,Q)$ for the resultant with respect to degrees $(p,q)$ of $(P,Q) \in \R[x]_{\leqslant p}\times\R[x]_{\leqslant q}$.  
We denote by $\Disc^p(P)$ the \emph{discriminant} of $P$ with respect to degree $p$. We recall the following result (see \cite[Chapter 12]{Gelfand}).

\begin{proposition}\label{prop:GKZ}
 For every $p,q\in\mathbb{N}$, $P \in \mathbb{R}[x]_{\leqslant p}$, $Q \in \mathbb{R}[x]_{\leqslant q}$ and $A\in GL(2,\R)$, we have 
 \begin{align*}
     \Res^{p,q}(A^{*p}P,A^{*q}Q) &= \det(A)^{pq} \Res^{p,q}(P,Q),\\
     \mathrm{Disc}^{p}(A^{*p}P) &= \det(A)^{p(p-1)}\mathrm{Disc}^{p}(P).
 \end{align*}
\end{proposition}

\subsection{Extensions to multivariate polynomials}

CAD projection is calculated with respect to the last variable $x_n$. Hence we now extend the previous constructions to $n$-variate polynomials by acting on the last variable only. 
We denote the extended operators by the same symbol, and use bold $\textbf{x}$ to denote a $(n-1)$-tuple.

\begin{definition}
  If $f$ is a function in more than $n-1$ variables, for every ${\bf{x}_0}\in\R^{n-1}$, we define the \emph{evaluation} $E_{\bf{x}_0}(f)$ to be the function obtained from $f$ by assigning the value $\textbf{x}_0$ to the $(n-1)$-tuple of first variables. 
\end{definition}

\begin{definition} \mbox{}
\begin{enumerate}
  \item We denote by $\R[\textbf{x};x_n,y]_d$ the set of polynomials of degree $d$ in $x_n,y$, i.e. those $g\in\R[x_1,\ldots,x_n,y]$ such that $E_{\textbf{x}_0}(g)\in\R[x_n,y]_{d}$, for all $\textbf{x}_0\in \R^{n-1}$. 
  \item We denote by $\R[\textbf{x};x_n]_{\leqslant d}$ the set of polynomials of degree at most $d$ in $x_n$, i.e. those $P\in\R[x_1,\ldots,x_n]$ such that $E_{\textbf{x}_0}(P)\in\R[x_n]_{\leqslant d}$, for all $\textbf{x}_0\in \R^{n-1}$. The degree in $x_n$ of $P$ is defined by $\deg_{x_n}(P) = \min\left\{d \in \N \; | \; P \in \R[\textbf{x};x_n]_{\leqslant d}\right\}.$
\end{enumerate}
    
Using again the evaluation operators, we now define the extension of the operators defined so far. These extensions are all defined by first evaluating an $n$-variate or $(n+1)$-variate polynomial at a point $\textbf{x}_0$ of $\R^{n-1}$ to obtain a univariate or bivariate polynomial, and then using the corresponding operator defined in Section \ref{sec:projroots}.

\end{definition}

\begin{definition}
For $d \in \N, A\in GL(2,\R)$ and $T\in \{H_A^d$, $\iota_A^{*d}$, $R_A, A^{*d}$, $\Disc^p\}$ from Section \ref{sec:projroots}, we define the action of the extended operator $T$ on $f$ in $\R[\textbf{x};x_n]_{\leqslant d}$ or $\R[\textbf{x};x_n,y]_d$ by defining the evaluations of $T(f)$ at every point $\textbf{x}_0\in\R^{n-1}$, setting $E_{\textbf{x}_0}\circ T=T\circ E_{\textbf{x}_0}$. For $p,q\in\N$, we extend similarly $\Res^{p,q}$ on $\R[\textbf{x};x_n]_{\leqslant p} \times \R[\textbf{x};x_n]_{\leqslant q}$.
\end{definition}

In the above definition, we described formally the usual pointwise extensions of operators. For example, writing $P({\bf{x}}, x_n) = \sum_{k=0}^{d_n} c_k({\bf{x}}) x_n^k$, we have
{\small{\begin{align*}
H^{d_n}(P)(\textbf{x}&, (x_n,y)) {=} E_{\textbf{x}}(H^{d_n}(P))(x_n,y)  
                                {=}  \sum_{k=0}^{d_n} c_k({\textbf{x}}) x_n^ky^{d-k},
\end{align*}}}
\vspace{-0.2cm}
{\small
\begin{equation}
       \!\! A^{*d_n}P(\textbf{x}, x_n)= \sum_{k=0}^{d_n} c_k(\textbf{x}) (a_{11}x_n+a_{12})^k(a_{21} x_n + a_{22})^{d_n-k} \label{eq:A*P}.
\end{equation}}
\!\!
Similarly, we extend the definitions of $\varphi_A, \mu_A, \iota_A$, keeping again the same symbol and acting on the last variable only. 

\begin{definition}
For $A\in GL(2,\R)$ and $F\in \{\varphi_A, \mu_A, \iota_A\}$ defined on the set $X$ as in Section \ref{sec:projline}, we define the extended function $F$ on $\R^{n-1} \times X$ by setting $F(\textbf{x}, u) = (\textbf{x}, F(u))$, i.e. the extended function is given by $\id \times F$ where $\id$ is the identity map of $\R^{n-1}$.
\end{definition}

Pointwise properties established in the previous sections readily generalize to the extended operators and functions defined above. However, it is important to fix a priori the degree with respect to which homogenizations, resultants and discriminants are computed. 

\begin{remark} One could also adopt another equivalent point of view to define these extensions to multivariate polynomials: first, define the operators of Section \ref{sec:projroots} in general for polynomials with coefficients in an integral domain and then, study whether or not these operators commute with ring homomorphisms (\cite[Section 4.2.1]{basu2007} for the resultant) and in particular with the evaluation $E_{\textbf{x}}$.
\end{remark}

We will analyze the behaviour of the order of polynomials with respect to homogenization and composition with a linear transformation. Note that the second part of \Cref{prop:order} justifies the definition of order of a binary form at a point of the projective line. 

\begin{proposition}\label{prop:order}\mbox{}
  \begin{enumerate}
    \item Let $f \in \R[x_1, \ldots, x_m]$ and $A \in GL(m,\R)$. We have 
    $\ord(f\circ A,v)=\ord(f,Av)$ for all $v\in \R^m$.
    \item If $g \in \R[\textbf{x};x_n,y]_d$, then we have $\ord(g,(\textbf{x},(x_n,y)))=\ord(g,(\textbf{x},k(x_n,y)))$ for every $\textbf{x} \in \R^{n-1}, (x_n,y)\in \R^2_*$ and $k\neq 0$.
    \item If $P \in \R[\textbf{x};x_n]_{\leqslant d}$. We have $\ord(P,(\textbf{x},x_n))=\ord(H^{d}(P),(\textbf{x},(x_n,1)))$ for every $\textbf{x}\in\R^{n-1}$, $x_n \in \R$.
  \end{enumerate}
\end{proposition}

The first assertion of \Cref{prop:order} is a particular case of a general theorem (\cite[Theorem 2.1]{McCallum1988}). We give an elementary proof of the particular case under consideration.

\begin{proof}
Using the chain rule, a simple induction shows that for $\forall k\in\N,i_1,\ldots,i_k\leqslant m$,
\[\partial_{i_1}\ldots\partial_{i_k}(f\circ A)=\sum_{j_1,\ldots,j_k=1}^mA_{j_1i_1}\cdots A_{j_ki_k}(\partial_{j_1}\ldots\partial_{j_k}f)\circ A.\]
 It follows from the definition of the order that $\ord(f\circ A,v)\geqslant \ord(f,Av)$, and the first result follows by symmetry. The second one follows by using $A=\id_{n-1}\times (k\id_2)$ and noticing that $g\circ A=k^dg$ has the same order as $g$. Finally, denote by $k$ the order of $P$ at $(\textbf{x}_0,u)$. Using the multi-index notation, we have $P(\textbf{x},x_n)=\sum_{|\alpha| + l \geqslant k}^d D_{\alpha,l} (\textbf{x}-\textbf{x}_0)^\alpha (x_n - u)^l$ and there exists $(\alpha_0,l_0)$ such that $|\alpha_0|+l_0 = k$ and $D_{\alpha_0,l_0}\neq 0$. We compute 
  \[H^{d}(P)(\textbf{x},(x_n,y)){=} \sum_{|\alpha| + l \geqslant k}^d D_{\alpha,l} (\textbf{x}-\textbf{x}_0)^\alpha (x_n - uy)^l y^{d-l}.\] 
  It is easily seen that $(\partial^\beta_{\textbf{x}} \partial^{a}_{x_n} \partial^b_{y}  H^{d}(P))$ vanishes at $(\textbf{x}_0,(u,1))$ whenever $|\beta|+a+b<k$ and that it is nonzero if $\beta=\alpha_0$, $a=l_0$ and $b=0$.
\end{proof}

%% file: proj-del.tex
\section{Projective delineability}
\label{SEC:MainDef}

In this section, we define the new concept of projective delineability of a polynomial over a subset $S\subset \R^{n-1}$, which extends the classical notion of delineability \cite{McCallum1988}  by allowing the root functions to reach infinity.

\begin{definition}
For $P\in \R[\textbf{x};x_n]$ of degree $d_n = \deg_{x_n}(P)$ in $x_n$ and $S\subset \R^{n-1}$, we define 
{\small{\begin{align*}
Z_{\R}(P,S) &= \{(\textbf{x}, x_n) \in S \times \R \; | \; P(\textbf{x}, x_n) = 0\} ,\\
Z_{\RP}(P,S) &= \{(\textbf{x}, (x_n: y)) \in S \times \RP \; |  H^{d_n}(P)(\textbf{x}, (x_n, y)) = 0\}.
\end{align*}
}}

We say that $x_n$ is a \emph{real root} of $P$ above $\textbf{x}_0\in S$ if $(\textbf{x}_0, x_n)\in Z_{\R}(P,S)$. Similarly, we say that $(x_n:y)$ is a \emph{projective root} of $P$ above $\textbf{x}_0$ if $(\textbf{x}_0, (x_n:y))\in Z_{\RP}(P,S)$. 
\end{definition}

As already mentioned, the projective roots of $P$ above $\textbf{x}_0$ consist of the real roots above $\textbf{x}_0$ (seen in the projective space) and the root at infinity above $\textbf{x}_0$. This latter root appears precisely when $\deg(E_{\textbf{x}_0}P)<d_n$ as explained in Lemma \ref{lemma:factorization} and the subsequent discussion.

\begin{definition}\label{def:projdel}
    We say that $P \in \R[\textbf{x};x_n]$ is \emph{projectively delineable} on a subset $S$ of $\R^{n-1}$ if there exist $k \in \N$ and some continuous functions $\theta_1, \ldots, \theta_k: S \to \RP$ such that 
        \begin{enumerate}
            \item $Z_{\RP}(P,S) = \text{Graph}(\theta_1) \sqcup \ldots \sqcup \text{Graph}(\theta_k)$;
            \item for all $l \in \{1,\ldots, k\},$ there exists $m_l \in \N^*$ such that for all $\textbf{x} \in S$, the multiplicity of the projective root $\theta_l(\textbf{x})$ of $E_{\textbf{x}}P$ (w.r.t. $\deg_{x_n}(P)$) is $m_l$. 
        \end{enumerate}
        We say that $\theta_1,\ldots,\theta_k$ are the \emph{projective root functions} of $P$ over $S$. 
        The graph of $\theta_l$ is a \emph{projective} $P$-\emph{section} over $S$.
\end{definition}

Requiring analyticity for $S$ and the root functions leads to the notion of analytic projective delineability, extending the notion of analytic delineability from \cite{McCallum1988}.

\begin{example}\label{ex:cub-hyp1} Polynomial $P(x_1,x_2) = (x_1x_2-1)(x_2+x_1^3)$ is projectively delineable on $\R$. Indeed, $Z_{\RP}(P,\R)$ is exactly the disjoint union of the graphs of the continuous functions $\theta_1, \theta_2 : \R \to \RP$ defined by $\theta_1(x_1) = (1 : x_1)$ and $\theta_2(x_1) = (-x_1^3 : 1)$ (respectively depicted in green and red in \Cref{fig:wrap}, top). Furthermore, $\theta_1(x_1)$ and $\theta_2(x_1)$ are both simple roots of $E_{x_1}(P)$ for all $x_1 \in \R$. 
\end{example}

In general, the delineability of $P$ on $S$ does not imply the projective delineability of $P$ on $S$, nor in the other direction. E.g. consider  the polynomials  $P(x_1,x_2) = x_1^2x_2^2 + 1$ and $Q(x_1,x_2) = x_1x_2 -1$ and $S = \R$. The relations between these two concepts are investigated in the following.
 
\begin{proposition}\label{prop:linkDelProjdel}
Let $S \subset \R^{n-1}$ and $P\in \R[\textbf{x};x_n]$ of degree $d_n$ in $x_n$ be defined by $P({\bf{x}}, x_n) = \sum_{k=0}^{d_n} c_k(\textbf{x}) x_n^k$.  
If there exists $k \in \{0,\ldots,d_n\}$ such that for all $\textbf{x} \in S$, we have
\begin{equation}\label{eqn:etoile}
      c_{d_n-k}(\textbf{x})\neq 0 \;\land\; \forall j \in \N : 0 \leq j \leq k-1, c_{d_n-j}(\textbf{x}) = 0,
\end{equation}
then $P$ is projectively delineable on $S$ if and only if $P$ is delineable on $S$.
\end{proposition}

\begin{remark}\label{rem:etoile}
$\eqref{eqn:etoile}$ means that for all $\textbf{x} \in S$, $P$ admits $\infty$ as a projective root of multiplicity $k$ (w.r.t. $d_n$) above $\textbf{x}$.
\end{remark}

\begin{proof}
 If $\eqref{eqn:etoile}$ holds and $P$ is delineable on $S$, then $Z_{\R}(P,S)$ is the disjoint union of the graphs of the continuous functions $\eta_1, \ldots, \eta_p : S \to \R$, and therefore
    {\small{$$Z_{\RP}(P,S) = \begin{cases}
        \bigsqcup_{i=1}^p \text{Graph}(\varphi \circ \eta_i) &\text{if $k = 0,$}\\
        \bigsqcup_{i=1}^p \text{Graph}(\varphi \circ \eta_i) \sqcup \text{Graph} (\theta_{p+1}) &\text{if $k \neq 0,$}
    \end{cases}$$}}
 \kern-0.65em where $\theta_{p+1}$ is constantly equal to $\infty$ on $S$ and $\varphi$ is defined in \eqref{eqn:iota}. Hence $P$ is projectively delineable since the condition on multiplicities is satisfied. 
 
 If $\eqref{eqn:etoile}$ holds and if $P$ is projectively delineable on $S$, then $Z_{\RP}(P,S)$ is the disjoint union of the graphs of the continuous functions $\theta_1, \ldots, \theta_p : S \to \RP$. If $k=0$, there is no root at infinity. We thus have 
 $Z_{\R}(P,S) = \bigsqcup_{i=1}^p \text{Graph}(\varphi^{-1} \circ \, \theta_i)$. Hence $P$ is delineable since the condition on multiplicities is satisfied by Lemma \ref{lemma:factorization}. If $k>0$, we introduce new projective root functions $\theta'_1,\ldots,\theta'_p$ of $P$ on $S$ such that one of them is constantly equal to $\infty$ in the following way. Considering for all $i\in\{1,\ldots,p\}$ the closed set $C_i=\theta_i^{-1}(\{\infty\})$, we have $S=\bigsqcup_{i=1}^p C_i$.  We pick $i_0\in\{1,\ldots,p\}$ such that $C_{i_0}\neq\emptyset$ and define the new functions by $\theta'_j|_{C_i}=\theta_{\sigma_i(j)}|_{C_i} ,$
  for all $i, j \in \{1,\ldots, p\}$,
  where $\sigma_i$ is the transposition $(i \; i_0)$ if $i \neq i_0$ and the identity if $i = i_0$. It follows from the pasting lemma that these functions are continuous. The definition ensures that $\theta'_{i_0}$ is constantly equal to $\infty$ and $Z_{\RP}(P,S)=\bigsqcup_{i = 1}^p \text{Graph}(\theta'_i)$. We thus have $Z_{\R}(P,S)=\bigsqcup_{i = 1, i\neq i_0}^p \text{Graph}(\varphi^{-1}\circ\theta'_i).$
  We show that for $i \in \{1, \ldots, p\}\setminus\{i_0\}$, the multiplicity of $\theta_i'(\textbf{x})$ as a projective root of $E_{\textbf{x}}P$ (w.r.t. $d_n$) is equal to that of $\theta_i(\textbf{x})$ and hence, does not depend on $\textbf{x} \in S$. If $C_i=\emptyset$, then $\theta_i' = \theta_i$ on $S$ and the conclusion is direct.
  Otherwise, the mutliplicity of $\theta_i(\textbf{x})$ is $k$ for every $\textbf{x}\in C_i$, hence for every $\textbf{x}\in S$. The multiplicity of $\theta_i'(\textbf{x})$ is that of $\theta_i(\textbf{x})$ or that of $\theta_{i_0}(\textbf{x})$, and both are equal to $k$. By Lemma \ref{lemma:factorization}, $P$ is therefore delineable on $S$.
\end{proof}

If we only assume the sign-invariance of the leading coefficient of $P$ on $S$, then the delineability of $P$ on $S$ does not imply the projective delineability of $P$ on $S$. A counter example is given by the polynomial $P(x_1,x_2,x_3) = x_1 x_3^3 + (x_1^2 + x_2^2)x_3^2 + 1$ over the line $S \equiv x_1 = 0$.
The converse is true if we suppose that $S$ is connected.
\begin{corollary}\label{cor:linkDelProjdel}
    Using the notation of Proposition \ref{prop:linkDelProjdel}, 
    \begin{enumerate}
        \item if $c_{d_n}$ never vanishes on $S$, then $P$ is projectively delineable on $S$ if and only if $P$ is delineable on $S$;
        \item if $c_{d_n}$ vanishes identically on a connected $S$ and $P$ is projectively delineable on $S$, then $P$ is delineable on $S$.
    \end{enumerate}
\end{corollary}

\begin{proof}
    The first item is the particular case $k = 0$ of Proposition \ref{prop:linkDelProjdel}. For the second item, the projective delineability of $P$ on the connected set $S$  implies that there is a projective root function, with a constant multiplicity $k$ on $S$, that is identically equal to $\infty$ on $S$. By Remark \ref{rem:etoile}, Condition \eqref{eqn:etoile} is then satisfied. From Proposition \ref{prop:linkDelProjdel} we deduce the delineability of $P$ on $S$.
\end{proof}

Corollary \ref{cor:linkDelProjdel} ($2$) fails without the connectedness assumption. A counter example is given by the polynomial 
$$P(x_1,x_2) = (x_1x_2-1)\left((x_1-1)x_2-1\right)^2$$
and the set $S = \{0,1\}$.

While projective delineability is not a theoretically weaker property than delineability, the corresponding projection  (discriminants and resultants only, see below) 
that ensures projective delineability is a proper subset of the reduced McCallum projection to ensure delineability \cite{Brown2001}.  Thus projective delineability offers computational savings.

\section{Local Result}

We prove a first projective analogue of the classical results concerning delineability \cite{collins1975, McCallum1998, Brown2001}. This result is local since it only holds in a suitable neighbourhood of each point in a connected analytic submanifold of $\R^{n-1}$. 

\begin{theorem}\label{thrm:local}
    Consider $S$ a connected analytic submanifold of $\R^{n-1}$ and  $P \in \R[\textbf{x};x_n]$ a polynomial of degree $d_n$ in $x_n$ that never nullifies on $S$, such that $\text{Disc}_{x_n}^{d_n}(P)$ is not the zero polynomial and is order-invariant on $S$. For each $\textbf{s} \in S$, there exists a connected neighbourhood $N_{\textbf{s}}$ of $\textbf{s}$ in $S$ such that $P$ is projectively delineable on $N_{\textbf{s}}$ and $H^{d_n}(P)$ is order-invariant on each projective $P-$section on $N_{\textbf{s}}$.
\end{theorem}
The idea of the proof is to use Theorem 2 of \cite{McCallum1998} for delineability. When the degree in $x_n$ of $E_{\textbf{x}}P$ is not constant on $S$, this result does not apply, and the multiplicity of $\infty$ as a projective root of $E_{\textbf{x}}P$ is not constant on $S$. 
Intuitively, we bypass this issue by using a geometric transformation to modify the zero set of the polynomial $P$ in order to guarantee that it does not reach infinity anymore, as shown in Example \ref{ex:cub-hyp2}. To do so, we replace $P$ by the polynomial $A^{*d_n}P$ on a neighbourhood of $\textbf{s} \in S$ for a well chosen matrix $A$ and ensure that this new polynomial has all the required properties (via Lemma \ref{lemma:key}). We then obtain locally the delineability of $A^{*d_n}P$ and deduce the desired result for $P$. 
\vspace{-0.2cm}
\begin{figure}[h]
    \centering
    \begin{tikzpicture}[scale=0.47]
        \begin{myaxis2}
            \draw[very thick, color=green100,domain=-4.5:4.5, variable=\x] plot ({\x},{-\x^3});
            \draw[very thick, color=red100,domain=0.2:4.5, variable=\x] plot ({\x},{1/\x});
            \draw[very thick, color=red100,domain=-4.5:-0.2, variable=\x] plot ({\x},{1/\x});
            \draw[very thick, color=blue100,domain=-4.5:4.5, variable=\x] plot ({\x},1);

            \draw[thick, color=black50, dashed] (-1,-10) -- (-1,10);
            \draw[thick, color=black50, dashed] (1,-10) -- (1,10);
        \end{myaxis2}

        \node at (12,2) {\includegraphics[width=15em,trim={1cm 3cm 4cm 3cm},clip]{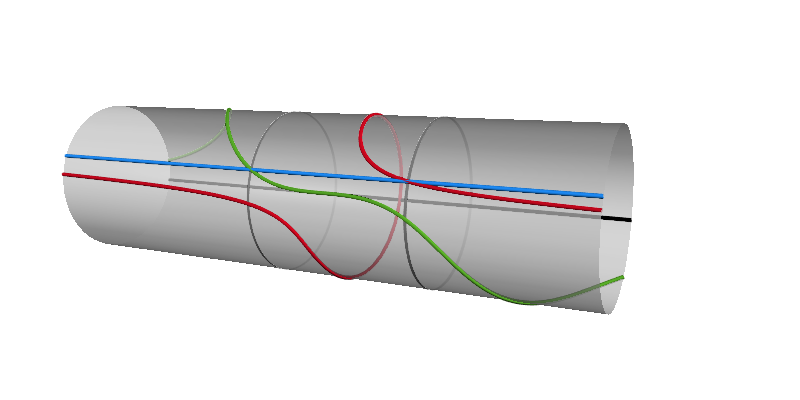}};
        \node[font=\footnotesize] at (10.5,.45)  {$-1$};
        \node[font=\footnotesize] at (11.75,.3)  {$0$};
        \node[font=\footnotesize] at (13,.15)  {$1$};
    \end{tikzpicture}
    \vspace{1em}
    \begin{tikzpicture}[scale=0.47]
        \begin{myaxis2}
            \draw[very thick, color=green100,domain=-4.5:-1.05, variable=\x] plot ({\x},{-\x^3/(1+\x^3)});
            \draw[very thick, color=green100,domain=-0.95:4.5, variable=\x] plot ({\x},{-\x^3/(1+\x^3)});
            \draw[very thick, color=red100,domain=-4.5:0.95, variable=\x] plot ({\x},{1/(\x-1)});
            \draw[very thick, color=red100,domain=1.05:4.5, variable=\x] plot ({\x},{1/(\x-1)});
            \draw[very thick,domain=-4.5:4.5, variable=\x] plot ({\x},-1);

            \draw[thick, color=black50, dashed] (-1,-10) -- (-1,10);
            \draw[thick, color=black50, dashed] (1,-10) -- (1,10);
        \end{myaxis2}

        \node at (12,2) {\includegraphics[width=15em,trim={1cm 3cm 4cm 3cm},clip]{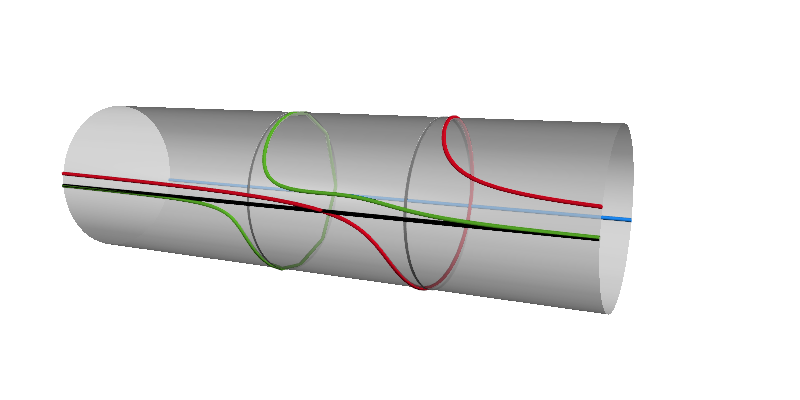}};
        \node[font=\footnotesize] at (10.5,.45)  {$-1$};
        \node[font=\footnotesize] at (11.75,.3)  {$0$};
        \node[font=\footnotesize] at (13,.15)  {$1$};
    \end{tikzpicture}

    \vspace{-2em}
    
    \caption{A visualization of $Z_{\RP}(P,\R)$ with the blue line $x_2 = 1$ (top) and $Z_{\RP}(A^{*2}P,\R)$ with the black line $x_2 = -1$ (bottom).}
    \label{fig:wrap}
    \vspace{-0.3cm}
\end{figure}
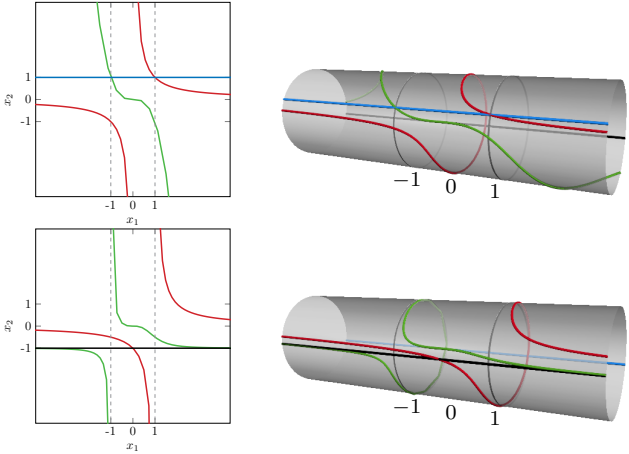
\begin{example}\label{ex:cub-hyp2}
Continuing Example \ref{ex:cub-hyp1} and setting $\textbf{s} = 0$, we observe that the root $\theta_1$ goes through $\infty$ exactly above $\textbf{s}$: this is precisely where its leading coefficient vanishes. If we set $A$ to the $2 \times 2$ lower triangular matrix of 1s, then \[A^{*2}P(x_1,x_2) = ((x_1-1) x_2 - 1) ((1+x_1^3)x_2 + x_1^3).\]
The leading coefficient of this polynomial vanishes exactly when $x_1 = \pm1$. In particular, on the open neighbourhood $N_{\textbf{s}} = (-1,1)$ of $\textbf{s}$, $A^{*2}P$ does not admit any root at infinity (see \Cref{fig:wrap}, bottom). The delineability of this polynomial is given by Theorem 2 of \cite{McCallum1998} (or can be checked by hand as in Example \ref{ex:cub-hyp1}). Denoting by $\eta_1$ and $\eta_2$ its real root functions on $N_{\textbf{s}}$, we can show that their compositions with $\varphi_A$ give the two projective root functions of $P$ on $N_{\textbf{s}}$.
\end{example}

\begin{lemma}\label{lemma:key}

Let $P\in \R[\textbf{x};x_n]$ be a nonzero polynomial of degree $d_n$ in $x_n$ and $A \in GL(2,\R)$.

\begin{enumerate} 
    \item The polynomial $A^{*d_n}P \in \R[\textbf{x};x_n]$ has degree less than or equal to $d_n$ in $x_n$ and its coefficient of degree $d_n$ in $x_n$ is $H^{d_n}(P)(\textbf{x}, (a_{11},a_{21}))$.
    \item The map $\mu_A$ is a self-homeomorphism of $\R^{n-1} \times \RP$ such that $\mu_A(Z_{\RP}(\AP)) = Z_{\RP}(P)$, preserving the multiplicities of roots.
    \item The map $\mu_A$ preserves the order. More precisely, 
    \begin{align*}
        \ord(H^{d_n}(A^{*d_n}P)&, (\textbf{x}, (x_n:y))) \\
        &=\ord(H^{d_n}(P), \mu_A(\textbf{x}, (x_n:y)))    
    \end{align*}
    for all $(\textbf{x}, (x_n: y)) \in \R^{n-1} \times \RP.$
    \item If $H^{d_n}(P)(\textbf{s}, (a_{11}, a_{21}))\neq0$ for some $\textbf{s} \in \R^{n-1}$, then the polynomial $A^{*d_n}P$ has degree $d_n$ and the set $V_{\textbf{s}} = \{\textbf{x} \in \R^{n-1} \; | \; H^{d_n}(P)(\textbf{x}, (a_{11}, a_{21}))\neq0\}$ is an open neighbourhood of $\textbf{s}$ such that $A^{*d_n}P$ has no root at infinity above $V_{\textbf{s}}.$
\end{enumerate}
\end{lemma}
\vspace{-0.4cm}
\begin{proof}\mbox{}
\begin{enumerate}
    \item This follows from Equation \eqref{eq:A*P}.
    \item The correspondence of the roots and multiplicities is discussed in Lemma \ref{lemma:linkMult} for univariate polynomials and directly extends to $n$-variate polynomials.
    \item By definition of order and Lemma \ref{lemma:pullback}, we need to show 
    \begin{align*}
        &\ord(R_A \circ H^{d_n}(P), (\textbf{x}, (x_n,y)))\\
        &\qquad = \ord(H^{d_n}(P), (\id \times A)(\textbf{x}, (x_n,y)))
    \end{align*}
     for all $(\textbf{x}, x_n, y) \in \R^{n+1}$.
     Furthermore, we have $R_A \circ H^{d_n}(P) =  H^{d_n}(P) \circ (\id \times A).$ The conclusion follows from the first item of Proposition \ref{prop:order}.
    \item The result concerning the degree of $A^{*d_n}P$ follows from the first assertion. The continuity of the map $\textbf{x} \mapsto H^{d_n}(P)(\textbf{x}, (a_{11}, a_{21}))$ implies that $V_\textbf{s}$ is an open neighbourhood of $\textbf{s}$. The leading coefficient (of degree $d_n$) of $A^{*d_n}P$ is $H^{d_n}(P)(\textbf{x}, (a_{11}, a_{21}))$ and therefore does not vanish above $V_\textbf{s}$. \qedhere
\end{enumerate}
\end{proof}

\begin{proof}[Proof of Theorem \ref{thrm:local}]
    Fix $\textbf{s} \in S$. Since $E_\textbf{s}(P)$ is not the zero polynomial, we can choose $(a_{11},a_{21}) \in \R^2_*$ such that $H^{d_n}(P)(\textbf{s}, (a_{11},a_{21})) \neq 0$ and then $a_{12}, a_{22} \in \R$ such that $\det(A)=1$.
    The fourth assertion of Lemma \ref{lemma:key} guarantees the existence of a neighbourhood $V_{\textbf{s}}$ of $\textbf{s}$ in $\R^{n-1}$ such that $A^{*d_n}P$ has no root at infinity above $V_{\textbf{s}}$. Then $N_{\textbf{s}}=V_{\textbf{s}}\cap S$ is an open subset of $S$ (hence an analytic submanifold), which we may assume to be connected, since $S$ is locally connected. We now apply the delineability result to $A^{*d_n}P$ above $N_{\textbf{s}}$ (Theorem 2 of \cite{McCallum1998}). To this aim, we first show that the hypotheses of this result are satisfied:  
    \begin{enumerate}
        \item The leading coefficient of $A^{*d_n}P$ (i.e. the coefficient of degree $d_n$) does not vanish on $N_{\textbf{s}}$.
        \item The discriminant $\Disc^{d_n}_{x_n}(A^{*d_n}P)$ equals $\Disc^{d_n}_{x_n}(P)$ by \Cref{prop:GKZ}. It is therefore not the zero polynomial and is order-invariant on $N_{\textbf{s}}$ by assumption. 
    \end{enumerate}
    Thus $A^{*d_n}P$ is delineable on $N_{\textbf{s}}$ and is order-invariant in each $A^{*d_n}P$-section over $N_{\textbf{s}}$, i.e. there exist continuous functions $\eta_1< \cdots < \eta_k: N_{\textbf{s}} \to \R$ such that:
    \begin{enumerate}
        \item $Z_{\R}(A^{*d_n}P, N_{\textbf{s}}) = \text{Graph}(\eta_1) \sqcup \ldots \sqcup \text{Graph}(\eta_k)$;
        \item the multiplicity of $\eta_l(\textbf{x})$ as a root of $E_{\textbf{x}}(A^{* d_n}P)$ is constant on $N_{\textbf{s}}$; and
        \item  the order of $A^{*d_n}P$ at $(\textbf{x}, \eta_l(\textbf{x})))$ is constant on $N_{\textbf{s}}$.
    \end{enumerate}
    Hence, by assertion 2 of Lemma \ref{lemma:key}, the continuous maps $\theta_l = \varphi_A \circ \eta_l = \mu_A \circ \varphi \circ \eta_l$ with $l \in \{1,\ldots, k\}$ are precisely the projective root functions of $P$.
    Indeed, we check that these functions satisfy the requirements of Definition \ref{def:projdel}:
    \begin{enumerate}
        \item For all $\textbf{x} \in N_{\textbf{s}}$, the points $\theta_1(\textbf{x}),\ldots,\theta_k(\textbf{x})$ are distinct by injectivity of $\mu_A \circ \varphi$.
        \item We have 
        $\varphi(Z_\R(A^{* d_n}P, N_{\textbf{s}})) = Z_{\RP}(A^{* d_n}P, N_{\textbf{s}})$ since $A^{* d_n}P$ has no root at infinity above $N_{\textbf{s}}$.  Hence, 
        \begin{align*}
            Z_{\RP}(P, N_{\textbf{s}}) &= \mu_A(Z_{\RP}(A^{* d_n}P, N_{\textbf{s}})) \\
            &= \text{Graph}(\theta_1) \cup \ldots \cup \text{Graph}(\theta_k).
        \end{align*}
        \item For all $\textbf{x} \in N_{\textbf{s}}$, $\varphi (\eta_l(\textbf{x}))$ is a projective root of $A^{* d_n}P$ of  multiplicity $m_l$ by Lemma \ref{lemma:factorization} and the subsequent discussion. Furthermore, Lemma \ref{lemma:key} guarantees that these multiplicities are preserved by $\mu_A$, that is $\theta_l(\textbf{x})$ is a projective root of multiplicity $m_l$ of $P$.
    \end{enumerate}
    The third assertion of Lemma \ref{lemma:key} and Proposition \ref{prop:order} give the conclusion concerning the order-invariance of $H^{d_n}(P)$ on the projective $P$-sections.
\end{proof}

\section{Global result and coverings}

In this section, we prove a global counterpart of Theorem \ref{thrm:local}, finding this requires some additional assumptions which are classical in the theory of covering spaces \cite{Fulton}. 
\begin{definition}\label{def:covering}
Let $X$ be a topological space. A \textit{covering space} of $X$ is a topological space $Y$ together with a continuous map $p: Y \to X$ such that, 
for each point $x \in X$, there exist an open neighbourhood $U$ of $x$ in $X$ and a non-empty family of disjoint open sets $(V_i)_{i\in I}$ of $Y$ such that $p^{-1}(U)=\sqcup_{i\in I}V_i$ and $p|_{V_i}:V_i\to U$ is a homeomorphism. The covering space $Y$ is trivial if these conditions hold for $U = X$.
\end{definition}

\begin{remark}
    Note that in this setting, defining for $i\in I$ the \emph{local section} $s_i=(p|_{V_i})^{-1}$ we obtain a family of continuous functions $(s_i: U \to Y)_{i \in I}$, such that $p \circ s_i = \text{id}_{U}$, $s_i(U)$ is open in $p^{-1}(U)$, and $p^{-1}(U)$ is the union of these sets.
\end{remark}
Recall that a topological space is simply connected if it is path connected (and thus connected), and for any $x\in X$, any loop of $X$ based at $x$ can be  continuously deformed (preserving the base point) onto the constant loop. 

\begin{theorem}\label{thrm:global}
Under the hypotheses of Theorem \ref{thrm:local}, and the further assumption that $S$ is simply connected, we have that $P$ is projectively delineable on $S$ and  $H^{d_n}(P)$ is order-invariant on each projective $P$-section over $S$.
\end{theorem}
\begin{proof}
    Theorem \ref{thrm:local} guarantees that the set $Z_{\RP}(P,S)$ is a covering space of $S$ for the usual (continuous) projection
        \[p: Z_{\RP}(P,S) \to S: (\textbf{x},(x_n:y)) \to \textbf{x}.\]
    Indeed, with the notation of Theorem \ref{thrm:local}, for all $\textbf{s} \in S$, projective delineability of $P$ over the open neighbourhood $N_{\textbf{s}}$ implies that $p^{-1}(N_{\textbf{s}})=Z_{\RP}(P, N_{\textbf{s}})$ is the disjoint union $\sqcup_{i=1}^k\text{Graph}(\theta_i)$ where $\theta_1,\ldots,\theta_k$ are the projective root functions from $N_{\textbf{s}}$ to $\RP$.
    Then for all $i \in \{1, \ldots, k\}, V_i=\text{Graph}(\theta_i)$ is homeomorphic to $N_{\textbf{s}}$ through $p$ and $\mathrm{Id}\times\theta_i$. Finally, since $\theta_i$ is a continuous function from $N_{\textbf{s}}$ to the Hausdorff space $\RP$, its graph is closed in $N_{\textbf{s}} \times \RP$, hence in $p^{-1}(N_{\textbf{s}}) \subset N_{\textbf{s}} \times \RP$. Thus $V_i=p^{-1}(N_{\textbf{s}})\setminus\cup_{j \neq i}V_j$ is open in  $p^{-1}(N_{\textbf{s}})$, and therefore in $Z_{\RP}(P,S)$. 

    Since $S$ is a simply connected analytic submanifold, any covering space over $S$ is trivial \cite[Corollary 13.8]{Fulton}, so there exist global sections $s_1, \ldots, s_k$ of $p: Z_{\RP}(P,S) \to S$. These functions read $s_i(\textbf{x})=(\textbf{x},t_i(\textbf{x}))$ and $t_i$ are by definition continuous projective root functions globally defined on $S$, such that $Z_{\RP}(P,S)$ is the disjoint union of their graphs. Finally for every $\textbf{s}\in S$ the function $t_i$ coincides over $N_{\textbf{s}}$ with one of the functions $\theta_1,\ldots\theta_k$. Therefore, the multiplicity of $t_i$ as a projective root of $E_{\textbf{s}}(P)$ is constant on $N_{\textbf{s}}$ as is the order of $H^{d_n}(P)$ at $t_i$ on $N_{\textbf{s}}$. We conclude by the connectedness of $S$. 
\end{proof}

\begin{remark}
    Observe that the projective root functions in Theorems \ref{thrm:local} and \ref{thrm:global} are analytic functions. Indeed, their local expressions in suitable (analytic) charts are exactly the root functions given by Theorem 2 of \cite{McCallum1998}.
\end{remark}
\noindent Theorem \ref{thrm:global} fails without the additional assumption:
\begin{proposition}\label{prop:counter}
    There exist a  connected analytic submanifold $S$ of $\R^{2}$ and  $P \in \R[x_1,x_2;x_3]_{\leqslant 4}$ such that 
    \begin{enumerate}
        \item $P$ is never nullified on $S$;
        \item the discriminant $\Disc_{x_3}^4(P)$ of $P$ is not the zero polynomial of $\R[x_1,x_2]$ and is order-invariant on $S$;
        \item $P$ is not projectively delineable over $S$.
    \end{enumerate}
\end{proposition}

\begin{proof}
    We define $S= \{(x_1,x_2) \in \R^2 \; | \; x_1^2 + x_2^2 - 1 = 0\}$ and $P\in\R[x_1,x_2;x_3]_{\leqslant 4}$ such that $P(x_1,x_2,x_3)$ is given by
    \[(1 - x_1)x_3^4  +4x_2 x_3^3  + (2 + 6 x_1)x_3^2  -4 x_2 x_3  + (1 - x_1).\]
    See Figure \ref{fig:counter}. Then $P$ satisfies the requested properties: it is never nullified on $S$ and we compute that
    \[\Disc^{4}_{x_3}(P)(x_1,x_2) = 2^{14} (-1 + x_1^2 + x_2^2)^2 (x_1^2 + x_2^2).\]
    Elementary computations show that it has order $2$ on $S$.
    
    To prove that $P$ is not projectively delineable over $S$, we first compute the set of projective roots of $P$ above every $\textbf{x}=(x_1,x_2)\in S$. If we write $\textbf{x}=(\cos(t),\sin(t))$ for $t\in\R$ then we can compute that $H^{4}(P)(\textbf{x},(x_3,y))$ is
        \[
       8\left(\cos\left(t/4\right)x_3 - \sin\left(t/4\right) y\right)^2 
       \left(\sin\left(t/4\right)x_3 + \cos\left(t/4\right) y\right)^2,
       \]
    implying that $P$ has exactly two projective roots over $(x_1,x_2)$ of multiplicity two given by 
        \[
        \left(\sin\left(t/4\right): \cos\left(t/4\right)\right) 
        \text{ and } \left(-\cos\left(t/4\right): \sin\left(t/4\right)\right).
        \]
        If we assume that $P$ is projectively delineable over $S$, then there exists two continuous functions $\theta_1, \theta_2: S \to \RP$ such that 
        $Z_{\RP}(P,S)=\text{Graph}(\theta_1) \sqcup \text{Graph}(\theta_2)$. 
        Hence, $Z_{\RP}(P,S)$ is not connected because the graphs $\text{Graph}(\theta_1)$ and $\text{Graph}(\theta_2)$ are non-empty and closed. However $Z_{\RP}(P,S)$ happens to be  connected as the image of $\R$ by the continuous function $\Upsilon : \R \to \R^2 \times \RP$ given by     
        \[ t \mapsto \left(\cos\left(t\right), \sin\left(t\right),\left(\sin\left(t/4\right): \cos\left(t/4\right)\right)\right). \qedhere \]
\end{proof}

 \begin{figure}[h]
    \hspace{1cm}\includegraphics[scale=0.35,trim={0cm 2cm 0cm 0cm},clip]{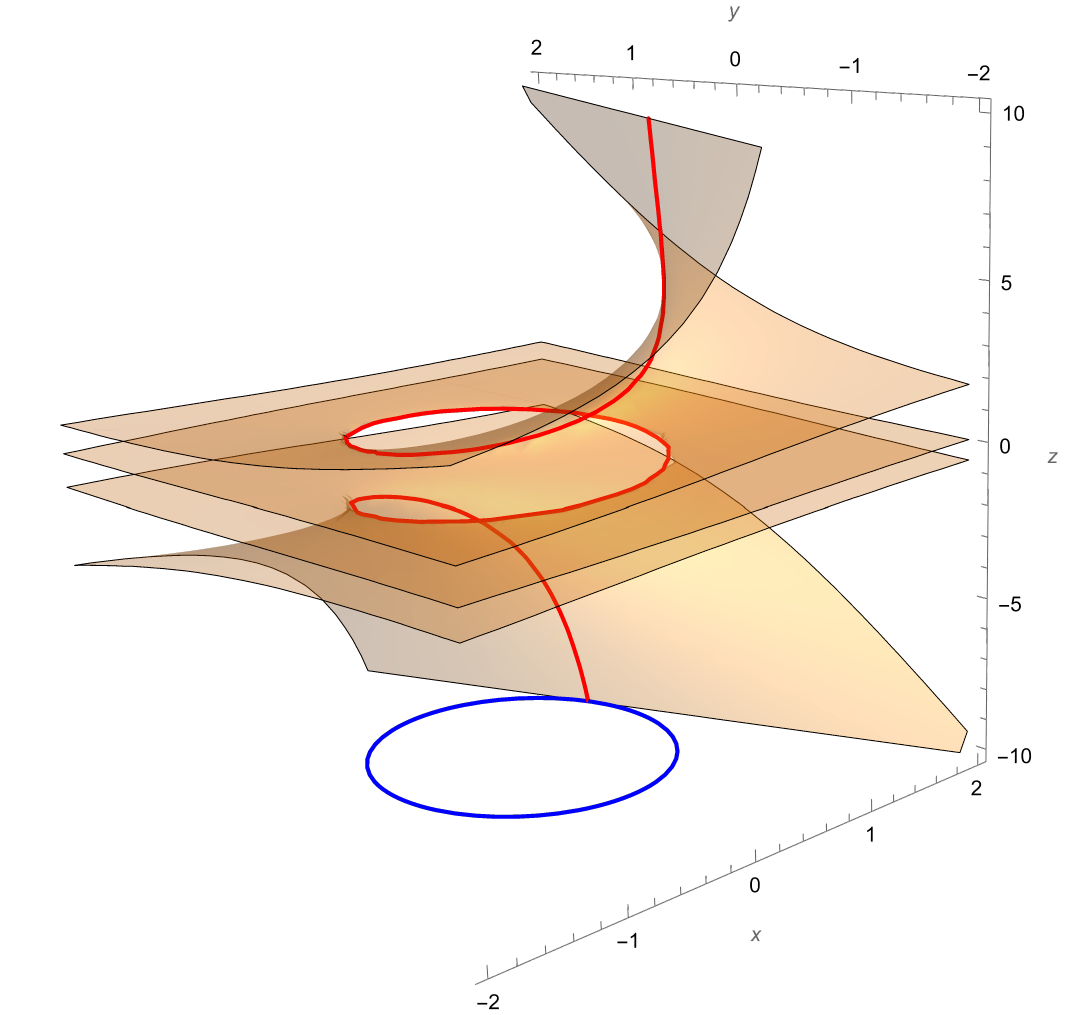}
    \caption{ $S$ (blue), $Z_{\R}(P, \R^2)$ (orange) and  $Z_{\R}(P,S)$ (red)}
    \label{fig:counter}
\end{figure}

We show that Theorem 2.2 of \cite{McCallum1999} extends naturally in the projective framework. In particular, the order-invariance of the resultant of two projectively delineable polynomials on $S$ guarantees the projective delineability of their product on $S$, allowing the notion to generalise from one polynomial to a set, as with classical delineability.

\begin{theorem}
    Let $S$ be a connected analytic submanifold of $\R^{n-1}$ and $P,Q \in \R[\textbf{x};x_n]$ of respective positive degree $p_n$ and $q_n$ in $x_n$. Assume $P, Q$ are never nullified on $S$ and $P$ is projectively delineable over $S$. If $\Res_{x_n}^{p_n, q_n}(P,Q)$ is not the zero polynomial and is order-invariant over $S$, then along each projective $P$-section over $S$, either $H^{q_n}(Q)$ vanishes completely or it never vanishes at all.
\end{theorem}

\begin{proof}
    By assumption, $P$ is projectively delineable over $S$ and we denote by $\theta_1, \ldots, \theta_k : S \to \RP$ its projective root functions over $S$. For every $l \in \{1,\ldots,k\}$, we consider 
        \[\Omega_l = \{\textbf{x} \in S \; | \; (\textbf{x},\theta_l(\textbf{x})) \in Z_{\RP}(Q,S)\}\]
    and we have to prove that this set is either empty or equal to $S$. We do it by using similar arguments as in Theorem \ref{thrm:local} and the connectedness of $S$. We may suppose that $\Omega_l$ is non-empty and take $\textbf{s}$ in its closure. 

    Since $E_\textbf{s}(PQ) = E_\textbf{s}(P)E_\textbf{s}(Q)$ is nonzero, we can choose $(a_{11},a_{21}) \in \R^2_*$ such that $H^{p_n}(P)H^{q_n}(Q)(\textbf{s}, (a_{11},a_{21})) \neq 0$ and then $a_{12}, a_{22} \in \R$ such that $\det A=1$. We then obtain a neighbourhood $N_{\textbf{s}}$ of $\textbf{s}$ as in the beginning of the proof of Theorem \ref{thrm:local} (but for the polynomial $PQ$). The projective root functions of $P$ above $N_{\textbf{s}}$ are given by the restrictions of $\theta_i, i \in \{1, \ldots, k\}$ to $N_{\textbf{s}}$. By Lemma \ref{lemma:linkMult} the projective root functions of $A^{*p_n}P$ above $N_{\textbf{s}}$
    are given by the continuous functions $\mu_{A^{-1}}\circ \theta_i$. Hence $A^{*p_n}P$ is projectively delineable over $N_{\textbf{s}}$ and has, by the choice of $A$, no root at infinity. By Proposition \ref{prop:linkDelProjdel}, we obtain the delineability of $A^{*p_n}P$ and denote by $\eta_i = \varphi^{-1} \circ \mu_{A^{-1}} \circ \theta_i, i \in \{1, \ldots, k\}$ its root functions.
    We now check the hypotheses of Theorem 2.2 of \cite{McCallum1999}, with $A^{*p_n}P$ and $A^{*q_n}Q$:
    \begin{enumerate}
        \item They have respective degree $p_n$ and $q_n$ in $x_n$ by the choice of $A$. Hence their degrees are positive.
        \item By Proposition \ref{prop:GKZ}, $\Res^{p_n,q_n}(A^{*p_n}P,A^{*q_n}Q)=\Res^{p_n,q_n}(P,Q)$, which is therefore not the zero polynomial and order-invariant over the connected analytic submanifold $N_{\textbf{s}}$.
    \end{enumerate}
    In particular, we conclude that $A^{*q_n}Q$ is sign invariant on the graph of $\eta_l$ over $N_{\textbf{s}}$. Since $\textbf{s}$ is in the closure of $\Omega_l$, there exists $\textbf{s}_0 \in N_{\textbf{s}} \cap \Omega_l$. Then $\theta_l(\textbf{s}_0)$ is a projective root of $Q$ above $\textbf{s}_0$, so $\mu_{A^{-1}}\circ\theta_l(\textbf{s}_0)$ is a projective root of $A^{*q_n}Q$ above $\textbf{s}_0$ and finally $A^{*q_n}Q(\textbf{s}_0,\eta_l(\textbf{s}_0))=0$. By the sign invariance $A^{*q_n}Q(\textbf{x},\eta_l(\textbf{x}))=0$ for every $\textbf{x}\in N_{\textbf{s}}$, and it follows that $(\textbf{x},\theta_l(\textbf{x})) \in Z_{\RP}(Q,S)$ for every $\textbf{x}\in N_{\textbf{s}}$. We just showed that $N_{\textbf{s}}\subset\Omega_l$ for all $\textbf{s}$ in the closure of $\Omega_l$, ensuring that the closure of $\Omega_l$ is a subset of the interior of $\Omega_l$. Then $\Omega_l$ is open and closed and therefore $\Omega_l = S$ by connectedness.
\end{proof}